\documentclass[12pt,reqno]{amsart}
\advance\hoffset by -0.5 truecm
\usepackage{amssymb}
\newtheorem{Theorem}{Theorem}[section]
\newtheorem{Lemma}[Theorem]{Lemma}

\newtheorem{Corollary}[Theorem]{Corollary}

\newtheorem{Proposition}[Theorem]{Proposition}

\newtheorem{Example}[Theorem]{Example}
\newtheorem{Remark}[Theorem]{Remark}

\begin{document}
\title[On the Yamabe constants of product manifolds]{On the Yamabe constants of product manifolds}

\author[C. Sung]{Chanyoung Sung}
 \address{Dept. of Mathematics Education  \\
          Korea National University of Education \\
          Cheongju, Korea
          }
\email{cysung@kias.re.kr}

\keywords{Yamabe constant, fiberwise symmetrization}

\subjclass[2020]{53C21, 58E99}

\date{}

%\maketitle

\begin{abstract}
By using the fiberwise spherical symmetrization we give a comparison theorem of Yamabe constants on warped products and prove the existence of radially-symmetric Yamabe minimizers on Riemannian manifolds given by products with round spheres.
\end{abstract}

\maketitle

%\tableofcontents

\section{Introduction}
%$g_V\emph{g}_V\textrm{g}_V\textbf{g}_V\textsf{g}_V\texttt{g}_V\textsl{g}_V\textsl{g}_V\textsc{g}_V\textsl{g}_V\textsl{g}_V\frak{g}_V$
%\begin{Remark}
%$g_V\emph{g}_V\textrm{g}_V\textbf{g}_V\textsf{g}_V\texttt{g}_V\textsl{g}_V\textsl{g}_V\textsc{g}_V\textsl{g}_V\textsl{g}_V\frak{g}_V$
%\end{Remark}
On a smooth closed Riemannian manifold $(M ,g)$ of dimension $m\geq 2$, its conformal class $[g]$ is defined as $$\{\phi^2g|\phi:M\stackrel{C^\infty}{\rightarrow}\Bbb R\backslash\{0\}\}$$ and the Yamabe constant $Y(M,[g])$ of $[g]$ is defined as
$$\inf_{\tilde{g}\in [g]}\frac{\int_Ms_{\tilde{g}}\ d\mu_{\tilde{g}}}{(\int_Md\mu_{\tilde{g}})^{\frac{m-2}{m}}}.$$ Here and henceforth
$s_{\tilde{g}}$ denotes the scalar curvature of $\tilde{g}$.
When $m\geq 3$, by writing $\tilde{g}\in [g]$ as $\tilde{g}=f^{\frac{4}{m-2}}g$, $Y(M,[g])$ can be expressed as
$$\inf\{Y_g (f)|f\in L^2_1 (M)\backslash\{0\}\}$$ where the Yamabe functional $Y_g$ is defined by
$$Y_g (f) := \frac{\int_M (a_m|df|_g^2 + s_g f^2) \ d\mu_g }{ {\left( \int_M |f|^{p_m}\ d\mu_g \right) }^{\frac{2}{p_m}}}$$
for $a_m = \frac{4(m-1)}{m-2}$ and $p_m=\frac{2m}{m-2}$.
It is a fundamental result
%proved in several stages by Yamabe \cite{Yamabe}, Trudinger \cite{Trudinger}, Aubin \cite{Aubin} and Schoen \cite{Schoen} is
that the functional $Y_g$ has a (Yamabe) minimizer $f_0$ which is smooth and positive, and the metric $f_0^{\frac{4}{m-2}} g$, called a Yamabe metric, has constant scalar curvature.
%Thus if $g$ is a Yamabe metric of $[g]$, then $Y(M,[g])$ is simply $s_g(\textrm{Vol}_g(M))^{\frac{2}{m}}$.
Moreover any metric of nonpositive constant scalar curvature \cite{Kobayashi} and any Einstein metric \cite{obata} are Yamabe metrics.

Aubin's inequality \cite{Aubin} states that $$Y(M,[g]) \leq Y(S^m,[g_{_{\Bbb S}}])$$
%=m(m-1)V_m^{\frac{2}{m}}$$
where $g_{_{\Bbb S}}$ is the standard round metric of constant curvature 1, so one can take the supremum of Yamabe constants over all conformal classes of $M$. The supremum value is called the Yamabe invariant $Y(M)$ of $M$, and gives a topological invariant of $M$ depending only on the smooth topology of $M$.

%The actual computation of the Yamave invariant is not easy, but there have been much progress in dimension 3 and 4. The tools for computing it in a general dimension have been developed, such as the gluing formulae of Yamabe invariants under the surgery of codimension 3 and more.

We are concerned with the product formulae of Yamabe constant. The exact product formula even in the case of a Riemannian product is tricky because of the nonlinearity of the Euler-Lagrange equation of the Yamabe functional. The first pioneering work was done by Petean :
\begin{Theorem}[Petean \cite{Petean-1}]
Let $(M^m,g)$ be a smooth closed Riemannian $m$-manifold with positive scalar curvature. Then there exists a constant $K>0$ depending only on $(M,g)$ such that for any smooth closed manifold $N$ of dimension $n$,
$$Y(N\times M)\geq K\ Y(S^{m+n}).$$
\end{Theorem}
He has also obtained an estimate of $K$ when $M=S^2$ or $S^3$, and $m+n\leq 5$ in cowork with Ruiz \cite{ruiz1, ruiz2}.
The next significant progress toward the product formula is made by Ammann, Dahl, and Humbert :
% whose result has not yet reached the complete form though.
\begin{Theorem}[Ammann, Dahl, and Humbert \cite{ADH}]
Let $(M^m,g)$ and $(N^n,h)$ be smooth closed Riemannian manifolds satisfying $m,n\geq 3$ and $Y(M,[g])\geq 0$.
 Then $$Y(N\times M)\geq a_{m,n} Y(M)^{\frac{m}{m+n}}Y(S^n)^{\frac{n}{m+n}}$$ where $a_{m,n}:=\frac{(m+n)a_{m+n}}{(ma_m)^{\frac{m}{m+n}}(na_{n})^{\frac{n}{m+n}}}$.
If $Y(N,[h])\geq 0$ and $\frac{s_g+s_h}{a_{m+n}}\geq \frac{s_g}{a_m}+\frac{s_h}{a_{n}}$, then $$Y(N\times M,[h+g])\geq a_{m,n}Y(M,[g])^{\frac{m}{m+n}}Y(N,[h])^{\frac{n}{m+n}}.$$
\end{Theorem}

While it seems hard to obtain a similar inequality for a general fiber bundle, one may try warped products first. We shall prove :
\begin{Theorem}\label{Yoon}
Let $(M^m,g)$ be a smooth closed Riemannian $m$-manifold with volume $V$ and Ricci curvature $\textrm{Ric}_g \geq m-1$ for $m\geq 2$, and $(N^n,h)$ be a smooth closed Riemannian $n$-manifold. Then for any smooth positive function $\rho$ on $N$
$$Y(N\times M,[h+\rho^2g]) \geq \left(\frac{V}{V_m}\right)^{\frac{2}{m+n}}Y(N\times S^m,[h+\rho^2g_{_{\Bbb S}}])$$ and if the equality holds then $(M,g)$ is Einstein with $\textrm{Ric}_g = m-1$.
\end{Theorem}
In this paper $V_m$ denotes the volume of $(S^m,g_{_{\Bbb S}})$, and $S^m_V:=(S^m,\textsl{g}_V)$ denotes an $m$-sphere with a metric of constant curvature and volume $V$.

The above spherical symmetrization technique can be also used to prove the existence of a symmetric minimizer of the Yamabe functional on manifolds with factors of $S^m$. A symmetric minimizer is meant by a minimizer invariant under an isometry subgroup of the given manifold, and it plays an important role not only in the equivariant Yamabe problem but also in studying the ordinary Yamabe problem. Unlike the first eigenfunctions, Yamabe minimizers are solutions of a nonlinear PDE so that such a symmetric minimizer may not exist if Yamabe constant is positive. Special interests have been given to manifolds with factors of $S^m$ \cite{Petean-3, jimmy-hector}, and the radial symmetry of a function on $S^m$  means its invariance under the standard action of $O(m)$ around the axis passing through the two poles. We shall prove :
\begin{Theorem}\label{newsfactory}
For a smooth closed Riemannian manifold $(N,h)$ and any positive constants $V_1,\cdots,V_k$, there exists a minimizer for the Yamabe functional of $(N\times \prod_{i=1}^kS^{m_i},h+\textsl{g}_{V_1}+\cdots +\textsl{g}_{V_k})$, which is radially symmetric in each factor of sphere $S^{m_i}$.
\end{Theorem}
%More applications are discussed in Section \ref{Apps}.
In the last two sections we shall provide different proofs of some well-known theorems by applying our result. One may also apply our method to the estimation of the CR Yamabe constant \cite{Sung-Takeuchi}, but we won't discuss it here.

\section{Preliminaries}

\subsection{Conventions and Notations}

Here are collected some common conventions and notations to be used throughout this paper.
First of all, every manifold is assumed to be smooth and connected unless otherwise stated.
%When $X$ is a $k$-dimensional Riemannian manifold (possibly with boundary), $\mu(X)$ denotes the $k$-dimensional volume of $X$. In a measure space $(X,\mu)$, $\mu(S)$ of a measurable subset $S$ denotes the measure of $S$, and we shall denote the characteristic function of a measurable subset $S$ by $\chi_{_S}$.

In case that a Lie group $G$ has a smooth left (or right) action on a smooth manifold $X$, it induces a left (or right) action on $C^\infty(X)$ respectively defined by
\begin{eqnarray}\label{BHCP}
(g\cdot F)(x):=F(g^{-1}\cdot x)\ \ \ \ \textrm{or}\ \ \ \ (F\cdot g)(x):=F(x\cdot g^{-1})
\end{eqnarray}
for $F\in C^\infty(X)$, $g\in G$ and $x\in X$.
%We also let $|G|$ denote the cardinality of $G$ for finite $G$ and if $G$ is infinite, it denotes the volume $\mu(G)$ of $G$ with respect to a Riemannian metric under consideration.

On a smooth Riemannian manifold $(M,g)$, $L_1^p(M)$ for $p\geq 1$ is the completion of $$\{f\in C^\infty(M)| \ ||f||_{L_1^p}:=(\int_M (|f|^p+|df|_g^p)\ d\mu_g)^{\frac{1}{p}}<\infty\}$$ with respect to $L_1^p$-norm $||\cdot||_{L_1^p}$.
%Given a function $F$ on a product manifold $N\times X$, we shall regard $F$ as a family of functions on $X$ parametrized by $N$, in which case we denote $F|_{\{s\}\times X}$ for $s\in N$ by $F_s$ for notational simplicity. We alert the reader that it should not be misunderstood as the abbreviated notation of the partial derivative $\frac{\partial F}{\partial s}$.

%Given a map $f$, its inverse image of a set $S$ is denoted by $f^{-1}S$, and the open ball of radius $\epsilon>0$ and center $p$ in any metric space is denoted by $B_\epsilon(p)$.

\subsection{Convergence of minimizing sequence}

The following is proved in \cite{Sung-weyl}, but we quote its proof for the reader.
\begin{Proposition}\label{yam}
On a smooth closed Riemannian $m$-manifold $(M,g)$ for $m\geq 3$, let $$\{\varphi_i\in L_1^2(M)|\int_M|\varphi_i|^{p_m}d\mu_g=1, i\in \Bbb N\}$$ be a minimizing sequence for $Y_g$, i.e. $\lim_{i\rightarrow \infty}Y_g(\varphi_i)= Y(M,[g])$. If
there exists $\tilde{\varphi}\in L^{p_m}(M)$ such that $|\varphi_i|\leq \tilde{\varphi}$ for all $i$, then there exists a Yamabe minimizer $\varphi$ to which a subsequence of $\{\varphi_i\}$ converges in $L_1^2$-norm.
\end{Proposition}
\begin{proof}
%Theorem 3.1 of our preceding work \cite{Sung-weyl}, but we quote it here for the reader.
Set $a:=a_m$ and $p:=p_m$.
%We may assume that $$\int_M|\varphi_i|^{p}d\mu_g=1$$ for all $i\in \Bbb N$ by considering $\frac{\varphi_i}{||\varphi_i||_{_{L^{p}}}}$.

From
\begin{eqnarray}\label{jongguk-0}
\lim_{i\rightarrow \infty}\int_M(a|d\varphi_i|^2+s_g\varphi_i^2)\ d\mu_g&=&Y(M,[g]),
\end{eqnarray}
there exists an integer $n_0$ such that if $i\geq n_0$ then
\begin{eqnarray*}
Y(M,[g])+1 &\geq& \int_M(a|d\varphi_i|^2+s_g\varphi_i^2)\ d\mu_g\\ &\geq& \int_M(a|d\varphi_i|^2-|\min_Ms_g|\ \varphi_i^2)\ d\mu_g\\ &\geq& \int_Ma|d\varphi_i|^2 d\mu_g-C(\int_M|\varphi_i|^p\ d\mu_g)^{\frac{2}{p}}
\end{eqnarray*}
for a constant $C>0$, implying that $$\sup_i\int_M|d\varphi_i|^2d\mu_g\leq C'$$ for a constant $C'>0$, and hence $\{\varphi_i|i\in \Bbb N\}$ is a bounded subset of $L_1^2(M)$.

Thus there exists $\varphi\in L_1^2(M)$ and a subsequence converging to $\varphi$ weakly in $L_1^2$, strongly in $L^2$, and pointwisely almost everywhere.
%\footnote{For the strong convergence in $L^2$, we used the Rellich-Kondrakov theorem which holds still on orbifolds. It can be easily derived by using the partition of unity. For a proof one may refer to \cite{Falsi}.}
By abuse of notation we let $\{\varphi_i\}$ be the subsequence.

Owing to that $|\varphi_i|\leq \tilde{\varphi}\in L^p(M)$, we can apply Lebesgue's dominated convergence theorem to obtain
\begin{eqnarray}\label{jongguk-1}
\int_Ms_g\varphi_i^2\ d\mu_g\rightarrow \int_Ms_g\varphi^2d\mu_g
\end{eqnarray}
 and
\begin{eqnarray}\label{jongguk-2}
\int_M|\varphi_i|^p d\mu_g\rightarrow \int_M|\varphi|^pd\mu_g,
\end{eqnarray}
 and hence $\int_M|\varphi|^pd\mu_g=1$ and there must exist $\lim_{i\rightarrow \infty}\int_M|d\varphi_i|^2d\mu_g$ by (\ref{jongguk-0}).

%In the same way as the previous proposition, the weak convergence $\varphi_i\rightarrow\varphi$ in $L_1^2$ and the strong convergence in $L^2$ dictate that
By the weak convergence $\varphi_i\rightarrow\varphi$ in $L_1^2$,
$$\int_M|d\varphi|^2d\mu_g+\int_M\varphi^2\ d\mu_g = \lim_{i\rightarrow \infty}(\int_M\langle d\varphi_i,d\varphi\rangle d\mu_g+\int_M \varphi_i\varphi\  d\mu_g),$$ so
\begin{eqnarray}\label{jongguk-3}
\int_M|d\varphi|^2d\mu_g &=& \lim_{i\rightarrow \infty}\int_M\langle d\varphi_i,d\varphi\rangle d\mu_g\\ &\leq& \limsup_{i\rightarrow \infty}(\int_M|d\varphi_i|^2d\mu_g    )^{\frac{1}{2}}(\int_M|d\varphi|^2d\mu_g)^{\frac{1}{2}}\nonumber\\ &=& \lim_{i\rightarrow \infty}(\int_M|d\varphi_i|^2d\mu_g    )^{\frac{1}{2}}(\int_M|d\varphi|^2d\mu_g)^{\frac{1}{2}}\nonumber
\end{eqnarray}
implying that
$$\int_M|d\varphi|^2d\mu_g \leq \lim_{i\rightarrow \infty}\int_M|d\varphi_i|^2d\mu_g.$$

Therefore
$$Y(M,[g])=\lim_{i\rightarrow \infty}\frac{\int_M(a|d\varphi_i|^2+s_g\varphi_i^2)\ d\mu_g}{(\int_M|\varphi_i|^p\ d\mu_g)^{\frac{2}{p}}}\geq
\frac{\int_M(a|d\varphi|^2+s_g\varphi^2)\ d\mu_g}{(\int_M|\varphi|^p\ d\mu_g)^{\frac{2}{p}}},$$ and hence $\varphi$ must be a Yamabe minimizer.

By (\ref{jongguk-1}) and (\ref{jongguk-2}), this implies that $$\lim_{i\rightarrow \infty}\int_M|d\varphi_i|^2d\mu_g=\int_M|d\varphi|^2d\mu_g.$$
Combined with (\ref{jongguk-3}), it gives
\begin{eqnarray*}
\lim_{i\rightarrow \infty}\int_M|d\varphi-d\varphi_i|^2d\mu_g&=&\lim_{i\rightarrow \infty}\int_M(|d\varphi|^2-2\langle d\varphi_i,d\varphi\rangle+|d\varphi_i|^2)\ d\mu_g=0,
\end{eqnarray*}
completing the proof that $\varphi_i\rightarrow\varphi$ strongly in $L_1^2$.

\end{proof}

\subsection{Fiberwise symmetrization}\label{elle}

For a Morse function $f$ on a smooth closed Riemannian $m$-manifold $(M^m,g)$ with volume $V$, the spherical rearrangement of $f$ is the radially-symmetric continuous function $f_*:S^m_V\rightarrow \Bbb R$  such that $\{y\in S^m_V| f_*(y)<t \}$ for any $t\in (\min(f), \max(f)]$ is the geodesic ball centered at the south pole $q_0$ with the same volume as $\{x\in M| f(x)<t \}$. By the construction $f_*$ is piecewise-smooth and satisfies that for any $p> 0$ and $(a,b)\subset \Bbb R$
\begin{eqnarray}\label{servant-now}
\int_{\{x\in M| a<f(x)<b \}}f(x)^p\ d\mu_g=\int_{\{y\in S^m_{V}|a<f_{*}(y)<b\}} f_{*}(y)^p\ d\mu_{\textsl{g}_V}
\end{eqnarray}
where $(\cdot)^p$ means $|\cdot|^p$ if $p$ is not an integer.
% where $S^m_V:=(S^m,\textsl{g}_V)$ denotes an $m$-sphere with a metric of constant curvature and volume $V$. We shall adopt these notations throughout the paper, and by the radial symmetry of a function on $S^m$ we always mean its invariance under the standard action of $O(m+1)$ around the axis passing through the two poles.

The natural question is whether this method of symmetrization can be done fiberwisely in case of a fiberwise Morse function on a fibred manifold. But it may not be possible to approximate a smooth function on a fiberd manifold by a fiberwise Morse function. Instead we consider a generic-fiberwise Morse function. Given a product manifold $N\times M$, we shall always consider $N$ as base and $M$ as fiber. A generic-fiberwise Morse function on $N\times M$ is a smooth function $F$ for which there exists an open dense subset $N_0$ of $N$ with  measure-zero complement such that $F$ restricted to each fiber over $N_0$ is a Morse function on $M$. We shall take the largest possible $N_0$, i.e. the union of all such $N_0$, and denote it by $N_0(F)$. In \cite{Sung}, we proved
\begin{Proposition}\label{morse-1}
Let $M$ and $N$ be smooth closed manifolds and $C^\infty(N\times M, \Bbb R)$ be endowed with $C^2$-topology. If $F\in C^\infty(N\times M, \Bbb R)$, then for any open neighborhood $U\subseteq C^\infty(N\times M, \Bbb R)$ of $F$ there exist $\tilde{F}\in U$ which is a generic-fiberwise Morse function on $M$.

Under the further assumption that a compact Lie group $G$ acts on $N$ and $M$ smoothly and $F:N\times M\rightarrow \Bbb R$ is invariant under this action, one can choose $\tilde{F}$ such that the $G$-orbit of $\tilde{F}$ also remains in $U$. If the action on $M$ is trivial, then $\tilde{F}$ can be chosen to be $G$-invariant.
\end{Proposition}

%Combining this with certain isoperimetric inequality, we obtained
%\begin{Theorem}[\cite{Sung}]
\begin{Theorem}
Let $(M^m,g)$ with volume $V$ and $(N^n,h)$ be smooth closed Riemannian manifolds. If $F:N\times M\stackrel{C^\infty}{\rightarrow} \Bbb R$ is a generic-fiberwise Morse function, then there exists a Lipschitz continuous function $F_{\bar{*}}\in L_1^p(N\times S^m_V)$ for any $p\geq 1$ which is smooth on an open dense subset with measure-zero complement such that for any $s\in N_0(F)$ $$F_{\bar{*}}|_{\{s\}\times S^m_V}=(F|_{\{s\}\times M})_*.$$
In addition, if a Lie group $G$ acts on $(M,g)$ isometrically and on $N$ both by smooth left actions, then for any $\mathfrak{g}\in G$
\begin{eqnarray}\label{Ocean-nuclear-1}
(\mathfrak{g}\cdot F)_{\bar{*}}=\mathfrak{g}\cdot F_{\bar{*}}
\end{eqnarray}
and the analogous statement holds for a right action too,
%\begin{eqnarray}\label{Ocean-nuclear}
%(F\cdot \mathfrak{g})_{\bar{*}}=F_{\bar{*}}\cdot \mathfrak{g}
%\end{eqnarray}
where the induced $G$ action on $N\times S^m_V$ acts trivially on $S^m_V$.

Moreover if Ricci curvature $Ric_g\geq m-1$, then
\begin{eqnarray}\label{JHK1}
\int_M| d^M F|^2_g\ d\mu_g\geq {\left( \frac{V}{V_m} \right) }^{\frac{2}{m}}\int_{S^m_{V}}| d^S F_{\bar{*}}|_{\textsl{g}_V}^2d\mu_{\textsl{g}_V}
\end{eqnarray}
\begin{eqnarray}\label{JHK2}
\int_M| d^N F|^2_h\ d\mu_g \geq  \int_{S^m_{V}}| d^{N} F_{\bar{*}}|^2_h \ d\mu_{\textsl{g}_V}.
\end{eqnarray}
%at any point of $N_0(F)$.
\end{Theorem}
In the above, we expressed the exterior derivative $d$ on $N\times M$ as $d^N+d^M$ where $d^N$ and $d^M$ denote the exterior derivative on each $N\times \{pt\}$ and each $\{pt\}\times M$ respectively, and similarly $d$ on $N\times S^m_V$ is split as $d^N+d^S$.

\section{Yamabe constant comparison and Proof of Theorem \ref{Yoon}}

When $m+n=2$, $N$ and $M$ must be a point and $S^2$ respectively.
The 2-sphere has the unique conformal class and its Yamabe constant is $8\pi$ by the Gauss-Bonnet theorem.
The desired inequality follows from $V\leq V_2$ which is the consequence of the Bishop-Gromov inequality.

We now consider the case of $m+n\geq 3$.
%This proof is almost parallel to that of Proposition \ref{Yoon-1}, so we shall use the same notations.
Set $a :=a_{m+n}$ and $p:=p_{m+n}$. By $|*|_\zeta$ we shall mean the norm of $*$ with respect to a metric $\zeta$. Let's define  $$\textbf{h}_V:=\left(\frac{V}{V_m}\right)^{\frac{2}{m}}h$$ so that
\begin{eqnarray}\label{JHK0}
\left( \frac{V_m}{V} \right)^{\frac{n}{m}}\int_{N}f\ d\mu_{\textbf{h}_V}=\int_Nf\ d\mu_h, \ \ \ \ \ \  \ \ \ \left(\frac{V}{V_m}\right)^{\frac{2}{m}}| d^{N}f|_{\textbf{h}_V}^2=| d^{N}f|_{h}^2
\end{eqnarray}
for any smooth $f:N\rightarrow \Bbb R$.

Let $\epsilon\ll 1$ be any small positive number and $\varphi: N\times M \rightarrow \Bbb R$ be a Yamabe minimizer of $(N\times M,h+\rho^2g)$.
%such that $\int_{N\times M}|\varphi|^p\ d\mu_{h+\rho^2g}=1$.
Take a smooth generic-fiberwise Morse function $\tilde{\varphi} :N\times M \rightarrow \Bbb R$ which is $C^2$-close enough to $\varphi$ so that
$$|Y_{h+\rho^2g} (\varphi)-Y_{h+\rho^2g}(\tilde{\varphi})|< \epsilon.$$

Using the curvature formula \cite{ON} of a warped product, and inequalities (\ref{JHK1}, \ref{JHK2}),
\begin{eqnarray}\label{Bjoring}
Y_{h+\rho^2g}(\tilde{\varphi}) &=&\frac{\int_{N\times M} \left[ a| d\tilde{\varphi}|_{h+\rho^2g}^2 + (s_h+\frac{1}{\rho^2}s_g+\frac{2m}{\rho}\Delta_h \rho-\frac{m(m-1)}{\rho^2}|d^N \rho|_h^2) \tilde{\varphi}^2\right] d\mu_{h+\rho^2g}}{ {\left(\int_{N\times M}
|\tilde{\varphi}|^p\ d\mu_{h+\rho^2g}\right) }^{\frac{2}{p}}}\nonumber\\
&\geq& \frac{\int_{N}\rho^m\int_M \left[ a(\rho^{-2}|d^M \tilde{\varphi}|_g^2+| d^N \tilde{\varphi}|_h^2) +(s_h+\cdots) \tilde{\varphi}^2\right] d\mu_gd\mu_h}{ {\left(\int_{N}\rho^{m}\int_M |\tilde{\varphi}|^p\ d\mu_gd\mu_h\right)}^{\frac{2}{p}}}\\
&\geq& \frac{\int_{N}\rho^m\int_{S^m_{V}} \left[ a(\frac{V}{V_m})^{\frac{2}{m}}(\rho^{-2}|d^S \tilde{\varphi}_{\bar{*}}|_{\textsl{g}_V}^2+|d^N \tilde{\varphi}_{\bar{*}}|_{\textbf{h}_V}^2) +(s_h+\cdots) \tilde{\varphi}_{\bar{*}}^2 \right] d\mu_{\textsl{g}_V}d\mu_{h}}{ {\left(\int_{N}\rho^{m}\int_{S^m_{V}}|\tilde{\varphi}_{\bar{*}}|^p\  d\mu_{\textsl{g}_V}d\mu_{h} \right) }^{\frac{2}{p}}},\nonumber
\end{eqnarray}
where $\cdots$ abbreviates $$\frac{m(m-1)}{\rho^2}+\frac{2m}{\rho}\Delta_h \rho-\frac{m(m-1)}{\rho^2}|d^N \rho|_h^2,$$
and $\int_M(s_h+\cdots) \tilde{\varphi}^2\ d\mu_g=\int_{S^m_V}(s_h+\cdots ) \tilde{\varphi}_{\bar{*}}^2\ d\mu_{\textsl{g}_V}$ is allowed due to (\ref{servant-now}) and the fact that $(s_h+\cdots )$ is constant in each fiber $\{s\}\times M$.

Applying $(\frac{V}{V_m})^{\frac{2}{m}}s_{\textbf{h}_V}=s_h$,  $(\frac{V}{V_m})^{\frac{2}{m}}\Delta_{\textbf{h}_V}=\Delta_h$, and (\ref{JHK0}),  the above estimation further leads to
\begin{eqnarray*}%\label{Pelopon}
&=&
\frac{\left( \frac{V_m}{V} \right)^{\frac{n}{m}}\int_{N}\rho^{m}\int_{S^m_{V}}(\frac{V}{V_m})^{\frac{2}{m}}\left[ a(\rho^{-2}|d^S \tilde{\varphi}_{\bar{*}}|_{\textsl{g}_V}^2 +| d^{N} \tilde{\varphi}_{\bar{*}}|_{\textbf{h}_V}^2)+(s_{\textbf{h}_V}+**) \tilde{\varphi}_{\bar{*}}^2\right] d\mu_{\textsl{g}_V}d\mu_{\textbf{h}_V}}{\left(\left( \frac{V_m}{V} \right) ^{\frac{n}{m}}\int_{N}\rho^{m}\int_{S^m_V}|\tilde{\varphi}_{\bar{*}}|^p\ d\mu_{\textsl{g}_V}d\mu_{\textbf{h}_V}\right)^{\frac{2}{p}}}
\\ &=& {\left( \frac{V}{V_m} \right) }^{\frac{2}{m+n}}Y_{\textbf{h}_V+\rho^2\textsl{g}_V} (\tilde{\varphi}_{\bar{*}} ) \\
&\geq& {\left( \frac{V}{V_m} \right)}^{\frac{2}{m+n}}Y(N\times S^m,[\textbf{h}_V+\rho^2\textsl{g}_V])\\
&=&  {\left( \frac{V}{V_m} \right)}^{\frac{2}{m+n}}Y(N\times S^m,[h+\rho^2g_{_{\Bbb S}}])
\end{eqnarray*}
where $**$ abbreviates
$$\left(\frac{V}{V_m}\right)^{\frac{2}{m}}\frac{m(m-1)}{\rho^2}+\frac{2m}{\rho}\Delta_{\textbf{h}_V} \rho-\frac{m(m-1)}{\rho^2}|d^N \rho|_{\textbf{h}_V}^2$$ so that
\begin{eqnarray*}
s_{\textbf{h}_V}+** &=& s_{\textbf{h}_V}+\frac{s_{\textsl{g}_V}}{\rho^2}+\frac{2m}{\rho}\Delta_{\textbf{h}_V} \rho-\frac{m(m-1)}{\rho^2}|d^N \rho|_{\textbf{h}_V}^2\\ &=& s_{\textbf{h}_V+\rho^2\textsl{g}_V}.
\end{eqnarray*}
This proves the inequality.
%and the first inequality is not an equality only because of $$\frac{m(m-1)}{\rho^2}\geq

Now let's consider the case when the equality holds. In that case, as $\tilde{\varphi}\rightarrow \varphi$ in $C^2$-norm,
$$Y_{h+\rho^2g}(\tilde{\varphi})\rightarrow Y(N\times M,[h+\rho^2g])=\left( \frac{V}{V_m}\right)^{\frac{2}{m+n}}Y(N\times S^m,[h+\rho^2g_{_{\Bbb S}}]).$$
This forces that $s_g\geq m(m-1)$ must be equal to $m(m-1)$, for otherwise the inequality in (\ref{Bjoring}) would be far from the equality for any $\tilde{\varphi}$ sufficiently $C^2$-close to $\varphi$. Combined with the fact that $Ric_g\geq m-1$, it follows that $Ric_g=m-1$.

\begin{Remark}\label{unicef}
When $N$ is a point, this inequality has been obtained by S. Ilias \cite{Ilias}, J. Petean \cite{Petean}, and N. Gamera, S. Eljazi, and H. Guemri \cite{Gamera} independently by using analogous methods. Another noteworthy fact in this case is :

When $N$ is a point (and $\rho\equiv 1$), the equality of the above theorem  holds if and only if $(M,g)$ is Einstein with $Ric_g=m-1$.

We only need to show the if part. If $Ric_g=m-1$, then $g$ is a Yamabe metric and hence
\begin{eqnarray*}
Y(M,[g])&=& m(m-1)V^{\frac{2}{m}}\\
&=& {\left( \frac{ V}{V_m} \right) }^{\frac{2}{m}} Y(S^m,[g_{_{\Bbb S}}])\\ &=& {\left( \frac{ V}{V_m} \right) }^{\frac{2}{m}} Y(S^m,[\textsl{g}_V]).
\end{eqnarray*}
%Conversely if $Y(M,[g]) ={\left( \frac{ V}{V_m} \right) }^{\frac{2}{m}} Y(S^m,[\textsl{g}_V])$, then there should exist $\tilde{\varphi}$ in the above proof such that $Y_{h+\rho^2g}(\tilde{\varphi})$ gets arbitrarily close to ${\left( \frac{V}{V_m} \right)}^{\frac{2}{m+n}}Y(N\times S^m,[h+\rho^2g_{_{\Bbb S}}])$ in the above proof while it is $C^2$-close to $\varphi$. Then from the first inequality of the above computation, $s_g\geq m(m-1)$ must be equal to $m(m-1)$. Combined with the fact that $Ric_g\geq m-1$, it follows that $Ric_g=m-1$.

%From the proof, it also follows that for a round sphere $(S^m, g_0)$ if a conformal metric $fg_0$ has constant curvature then so is $F_{\bar{*}}g_0$ for the spherical rearrangement $F_{\bar{*}}$ of $f$.

\end{Remark}

\section{Symmetric minimizers and Proof of Theorem \ref{newsfactory}}

Let's consider the case when a Lie group $G$ acts on $N$ from the left. Regarding $S^m$ as an unit sphere in $\Bbb R^{m+1}$, it has a standard $O(m)$-action leaving the last coordinate fixed so that a function on $S^m$ is invariant under this $O(m)$-action iff it is a radially-symmetric function in the geodesic coordinate around the south pole. This $O(m)$-action is also extended to $N\times S^m$ commuting with the $G$-action. Thus we have a smooth left action of $G\times O(m)$ on $N\times S^m$.
\begin{Lemma}
Let $(N,h)$ be a smooth closed Riemannian $n$-manifold and $\rho$ be a smooth positive function on $N$.
If a compact Lie group $G$ acts on $N$ smoothly from the left and $(N\times S^m, [h+\rho^2g_{_{\Bbb S}}])$ admits a Yamabe minimizer invariant under the $G$-action,
then it also admits a Yamabe minimizer invariant under $G\times O(m)$ where $O(m)$ acts on $S^m$ standardly as above.
\end{Lemma}
\begin{proof}
%Since $[h+\rho^2\textsl{g}_V]=[(\frac{V_m}{V})^{\frac{2}{m}}h+\rho^2g_{_{\Bbb S}}]$ as a conformal class of $N\times S^m$, we will consider the metric
Set $\hat{g}:=h+\rho^2g_{_{\Bbb S}}$ and $p:=p_{m+n}$, and define a $C^2$-norm on $C^2(N\times S^m,\Bbb R)$ using any metric.

Let $\varphi$ be a minimizer of $Y_{\hat{g}}$, which is invariant under the $G$-action and satisfies $\int_{N\times S^m}\varphi^p\ d\mu_{\hat{g}}=1$. By Corollary \ref{morse-1}, for each $k\in \Bbb N$ we can choose a smooth $G$-invariant generic-fiberwise Morse function  $\varphi_k:N\times S^m\rightarrow \Bbb R$ such that $$||\varphi_k-\varphi||_{C^2}<\frac{1}{k},\ \ \ \ \textrm{and}\ \ \ \ Y_{\hat{g}}(\varphi_k)\leq Y_{\hat{g}}(\varphi)+\frac{1}{k}.$$
%\int_{N\times S^m}\varphi_k^p\ d\mu_{\hat{g}}=1.

From the proof of Theorem \ref{Yoon}, we have
\begin{eqnarray*}
Y_{\hat{g}}((\varphi_k)_{\bar{*}}) &\leq& Y_{\hat{g}}(\varphi_k)\\ &\leq&  Y_{\hat{g}}(\varphi)+\frac{1}{k}\\ &=& Y(N\times S^m,[\hat{g}])+\frac{1}{k}
\end{eqnarray*}
implying that $\{(\varphi_k)_{\bar{*}}| k\in \Bbb N\}$ is a minimizing sequence for $Y_{\hat{g}}$, and $\{\frac{(\varphi_k)_{\bar{*}}}{||(\varphi_k)_{\bar{*}}||_{L^p}}|k\in \Bbb N\}$ is also a minimizing sequence satisfying $$\int_{N\times S^m}\left(\frac{(\varphi_k)_{\bar{*}}}{ ||(\varphi_k)_{\bar{*}}||_{L^p}}\right)^p\ d\mu_{\hat{g}}=1.$$
We also have the $C^0$-bound of this minimizing sequence : $$\sup_{k\in \Bbb N}\frac{||(\varphi_k)_{\bar{*}}||_\infty}{||(\varphi_k)_{\bar{*}}||_{L^p}}=\sup_{k\in \Bbb N}\frac{||\varphi_{k}||_\infty}{||\varphi_{k}||_{L^p} }<\infty,$$ because all $\varphi_k$ are $C^2$-close to $\varphi$. Thus Proposition \ref{yam} asserts that there exists a subsequence of $\{\frac{(\varphi_k)_{\bar{*}}}{||(\varphi_k)_{\bar{*}}||_{L^p}}|k\in \Bbb N\}$ converging to a Yamabe minimizer $\varphi_{\infty}\in L_1^2(N\times S^m)$ in $L_1^2$-norm.
% and hence also pointwisely a.e. by taking a further subsequence.

We claim that $(\varphi_k)_{\bar{*}}$ for any $k$ is $G$-invariant. Indeed for any $\mathfrak{g}\in G$
\begin{eqnarray*}
(\varphi_k)_{\bar{*}}&=& (\mathfrak{g}\cdot\varphi_k)_{\bar{*}}\\ &=& \mathfrak{g}\cdot(\varphi_k)_{\bar{*}}
\end{eqnarray*}
by using (\ref{Ocean-nuclear-1}).

Since each $(\varphi_k)_{\bar{*}}$ is also $O(m)$-invariant commuting with the $G$-action, it is $(G\times O(m))$-invariant.
Thus the limit $\varphi_{\infty}$ which is the a.e. pointwise limit of a subsequence of $\{(\varphi_k)_{\bar{*}}|k\in \Bbb N\}$ is a.e. $(G\times O(m))$-invariant.  Since a Yamabe minimizer $\varphi_{\infty}$ is smooth, it has to be everywhere $(G\times O(m))$-invariant.
%ntinuity of , should be $(G\times O(m))$-invariant on the whole $N\times S^m$.
\end{proof}

Applying the above lemma inductively, one can get Theorem \ref{newsfactory}.
\begin{Remark}
In \cite{Sung} we have shown that the fiberwise spherical symmetrization on a warped product produces a smaller Rayleigh quotient, thereby giving a comparison theorem of the first eigenvalues. Thus one can apply the above method to the first eigenfunction of the Laplacian on the warped product with round $m$-sphere to obtain a radially-symmetric eigenfunction. Indeed one can obtain the followings :

Let $(N,h)$ and $\rho$ be as in the above lemma. If a compact Lie group $G$ acts on $N$ smoothly from the left and $(N\times S^m, h+\rho^2g_{_{\Bbb S}})$ admits a $G$-invariant first eigenfunction for $\Delta$, then it also admits a first eigenfunction invariant under $G\times O(m)$ where $O(m)$ acts on $S^m$ standardly as above.

For any smooth closed Riemannian manifold $(N,h)$, there exists a first eigenfunction of $(N\times \prod_{i=1}^kS^{m_i},h+\textsl{g}_{V_1}+\cdots +\textsl{g}_{V_k})$, which is radially symmetric in each $S^{m_i}$.
In fact, the latter statement can be also proved by using a general fact
$$\lambda_1(X\times M,h+g)=\min(\lambda_1(X,h),\lambda_1(M,g))$$ and the fact that there exists a radially-symmetric first eigenfunction on a round sphere of any dimension.(Indeed such a spherical harmonic on $\Bbb S^{m}$ is unique up to constant multiplication. It is given by the restriction of a linear function
$$(x_1,\cdots,x_{m+1})\cdot q_0=-x_{m+1}$$ of $\Bbb R^{m+1}$ to $\Bbb S^{m}=\{(x_1,\cdots,x_{m+1})\in \Bbb R^{m+1}|\sum_ix_i^2=1\}$.)

%Our method can be also applied to study the Yamabe problem on noncompact product manifolds such as $N\times \Bbb R^m$ and $N\times \Bbb H^m$.

\end{Remark}

\section{Yamabe invariant of $S^1\times M$}

Applying Theorem \ref{Yoon} to the case when $N$ is a circle and $\rho\equiv 1$, we have
\begin{Corollary}
Let $M$ be as in Theorem \ref{Yoon}. Then $$Y(S^1\times M)\geq \left(\frac{V}{V_m}\right)^{\frac{2}{m+1}}Y(S^{m+1}).$$
\end{Corollary}
\begin{proof}
Applying Theorem \ref{Yoon} with $(N,h)=(S^1_T,\textsl{g}_T)$,
\begin{eqnarray*}
Y(S^1\times M)&\geq& Y(S^1\times M,[\textsl{g}_T+g])\\
&\geq& \left(\frac{V}{V_m}\right)^{\frac{2}{m+1}}Y(S^1\times S^m,[\textsl{g}_T+g_{_{\Bbb S}}]).
\end{eqnarray*}
Letting $T\rightarrow \infty$, the desired conclusion follows from R. Schoen's theorem \cite{Schoen2}
$$\lim_{T\rightarrow \infty}Y(S^1\times S^m,[\textsl{g}_T+g_{_{\Bbb S}}])=Y(S^{m+1}).$$
\end{proof}
\begin{Remark}
Petean \cite{Petean-2} already obtained this corollary by a different method which is also based on the isoperimetric inequality.
\end{Remark}

When $(M,g)$ is Einstein, Theorem \ref{Yoon} can be also expressed as
\begin{Corollary}\label{bestform}
Let $M$ and $N$ be as in Theorem \ref{Yoon}, and further suppose that $Ric_g=m-1$. Then  $$Y(N\times M,[h+\rho^2g]) \geq \left(\frac{Y(M,[g])}{Y(S^m)}\right)^{\frac{m}{m+n}}Y(N\times S^m,[h+\rho^2g_{_{\Bbb S}}])$$
and
$$Y(S^1\times M)\geq \left(\frac{Y(M,[g])}{Y(S^m)}\right)^{\frac{m}{m+1}}Y(S^{m+1}).$$
\end{Corollary}
\begin{proof}
When $Ric_g=m-1$, by Remark \ref{unicef} we have
$$Y(M,[g])={\left( \frac{ V}{V_m} \right) }^{\frac{2}{m}} Y(S^m,[g_{_{\Bbb S}}])={\left( \frac{ V}{V_m} \right) }^{\frac{2}{m}} Y(S^m).$$
Plugging $\frac{ V}{V_m}=\left( \frac{Y(M,[g])}{ Y(S^m)} \right)^{\frac{m}{2}}$ into the inequality of Theorem \ref{Yoon} and the above corollary gives the desired inequalities.
\end{proof}

\begin{Example}
It would be helpful to illustrate it by an example here.
As is well-known, the Fubini-Study metric $g_{_{FS}}$ of $\Bbb CP^2$ is an Einstein metric of positive Ricci curvature and in fact it realizes $Y(\Bbb CP^2)=12\sqrt{2}\pi$  by the result of LeBrun \cite{LeBrun}.

Letting $a>0$ be the constant such that the Ricci curvature of $a^2g_{_{FS}}$ is 3,
% and $V$ be the volume of $(\Bbb CP^2,a^2g_{FS})$. %Applying Corollary \ref{bestform} and Schoen's theorem again,
we have
\begin{eqnarray*}
Y(S^1\times \Bbb CP^2)&\geq& \lim_{T\rightarrow \infty}Y(S^1\times\Bbb CP^2,[\textsl{g}_T+a^2g_{_{FS}}])\\
&\geq&  \left(\frac{Y(\Bbb CP^2,[a^2g_{_{FS}}])}{Y(S^4)}\right)^{\frac{4}{5}}\lim_{T\rightarrow \infty}Y(S^1\times S^4,[\textsl{g}_T+g_{_{\Bbb S}}])     \\ &=&  \left(\frac{Y(\Bbb CP^2,[g_{_{FS}}])}{Y(S^4)}\right)^{\frac{4}{5}}Y(S^{5})\\ %&=& \left(\frac{12\sqrt{2}\pi}{8\sqrt{6}\pi}\right)^{\frac{4}{5}}Y(S^{5})\\
&=& \left(\frac{9}{16}\right)^{\frac{1}{5}}Y(S^5).
\end{eqnarray*}
It is an interesting question whether $Y(\Bbb CP^2\times S^1)$ is equal to $\left(\frac{9}{16}\right)^{\frac{1}{5}}Y(S^5)$ or not.
\end{Example}

\section{Optimal volume sphere theorem}
%\subsection{Miscellaneous byproducts}
As another application of Theorem \ref{Yoon}, we can give a different proof of the 3-dimensional case of the optimal volume sphere theorem problem \cite{Petersen},
which was proved by Hamilton's theorem \cite{hamilton} using the Ricci flow.
\begin{Theorem}\label{Chung}
Let $(M,g)$ be a smooth closed Riemannian 3-manifold with $\textrm{Ric}_g  \geq 2$ and volume $V> \frac{V_3}{2}$.
Then $M$ is diffeomorphic to $S^3$.
\end{Theorem}
\begin{proof}
This time let's apply Theorem \ref{Yoon} with $N$ equal to a point and $\rho\equiv 1$, and we get
\begin{eqnarray*}
Y(M,[g])&\geq& \left(\frac{V}{V_3}\right)^{\frac{2}{3}}Y(S^3_V,[\textsl{g}_V])\\
&>& \frac{Y(S^3_V,[\textsl{g}_V])}{2^{\frac{2}{3}}}\\
&=& \frac{Y(S^3)}{2^{\frac{2}{3}}}
\end{eqnarray*}
by the condition that $V> \frac{V_3}{2}$. By Akutagawa and Neves's theorem \cite{Neves}, any closed 3-manifold with Yamabe invariant greater than $\frac{Y(S^3)}{2^{\frac{2}{3}}}$ must be diffeomorphic to a connected sum $S^3\# j(S^2\times S^1)\# k(S^2\tilde{\times} S^1)$ for some $j,k\geq 0$, where  $S^2\tilde{\times} S^1$ is the non-orientable $S^2$-bundle over $S^1$. Since the fundamental group of $M$ is finite by Myers's theorem, $M$ must be $S^3$.
\end{proof}
%we have the following immediate applications. First, one can give a short proof of a special case of T. Colding's theorem \cite{Colding} which states that for any $\varepsilon>0$ there exists a $\delta(n,\varepsilon)>0$ such that any $n$-manifold $(M ,g)$ with  $Ric_g \geq n-1$ and volume $V\geq V_m-\delta$ satisfies that the Gromov-Hausdorff distance between $(M ,g)$ and $(S^m, g_{_{\Bbb S}})$ is at most $\varepsilon$.

\bigskip

\noindent{\bf Declarations}

\medskip

\noindent{\bf Data Availability} : Data sharing is not applicable to this article as no new data were created or analyzed in this study.

\noindent{\bf Conflicts of Interest} : The author has no relevant financial or non-financial interests to disclose.

\noindent{\bf Funding} : No funding has been provided for this study.

\vspace{0.5cm}

\end{document}